\documentclass{article}
\usepackage{amsthm}
\usepackage{amsfonts}
\usepackage[T1]{fontenc}
\usepackage{epsfig}
\usepackage[lflt]{floatflt}
\usepackage{epsfig}
\usepackage{graphicx}
\usepackage{amsmath}
\usepackage{amssymb}
\usepackage{color}
\usepackage{subfigure}
\usepackage[english]{babel}
\usepackage{empheq}
\usepackage{caption}
\usepackage{stmaryrd}
\usepackage{mathrsfs}
\newcommand{\al}{\alpha}
\newcommand{\dd}{{\rm d}}
\theoremstyle{plain}
\newtheorem{theorem}{Theorem}[section]

\newtheorem{proposition}[theorem]{Proposition}
\newtheorem{corollary}[theorem]{Corollary}
\theoremstyle{definition}

\newtheorem{remark}[theorem]{Remark}

\usepackage{makeidx}
\usepackage{latexsym}
\usepackage{enumerate}
\usepackage{anysize}
\marginsize{2.5cm}{2.5cm}{2.5cm}{3cm}

\begin{document}

\begin{center}\emph{}
\LARGE
\textbf{Explicit solutions related to the Rubinstein binary-alloy solidification problem with a heat flux or a convective condition at the fixed face}
\end{center}


\begin{center}
{\sc Lucas D. Venturato, Mariela B. Cirelli, Domingo A. Tarzia\\
 CONICET - Depto. Matem\'atica, FCE, Universidad Austral, Paraguay 1950, S2000FZF Rosario, Argentina \\
(LVenturato@austral.edu.ar, cirelli@fceia,edu.ar, DTarzia@austral.edu.ar)}
\vspace{0.2cm}
\end{center}
      
\small


\noindent \textbf{Abstract:} 
Similarity solutions for the two-phase Rubinstein binary-alloy solidification problem in a semi-infinite material are developed. These new explicit solutions are obtained by considering two cases: a heat flux or a convective boundary conditions at the fixed face, and the necessary and sufficient conditions on data are also given in order to have an instantaneous solidification process. We also show that all solutions for the binary-alloy solidification problem are equivalent under some restrictions for data. Moreover, this implies that the coefficient which characterizes the solidification front for the Rubinstein solution must verify an inequality as a function of all thermal and boundary conditions.




\section*{Nomenclature}
$\begin{array}{ll}
C_0 & \text{ initial concentration} \\
d & \text{ mass diffusion} \\
f_l & \text{ liquidus curve} \\
f_s & \text{ solidus curve} \\
h_0 & \text{ coefficient that characterizes the heat transfer condition at } x=0 \\
k & \text{ thermal conductivity} \\
q_0 & \text{ coefficient that characterizes the heat flux at } x=0 \\
s(t) & \text{ position of the solidification front at time } t \\
t & \text{ time} \\
T_0 & \text{ initial temperature  } \\
T_k & \text{ critical temperature of solidification} \\
\alpha & \text{ thermal diffusivity} \\
\gamma & \text{ latent heat of fusion by density of mass} \\
\lambda & \text{ constant which characterizes the moving boundary for heat flux boundary condition} \\
\delta & \text{ constant which characterizes the moving boundary for convective boundary condition} \\
\rho & \text{ mass density}
\end{array}$

\section*{Subscript}
$\begin{array}{ll}
s & \text{ solid phase} \\
l & \text{ liquid phase}
\end{array}$

\section{Introduction}
Heat transfer during the solidification of alloys has been of particular interest in many engineering applications, especially in the field of casting, welding, thermal energy storage systems, crystal growth in semiconductors.

A semi-infinite material of a binary alloy consisting of two components $ A $ and $ B $ is considered. Let $ C $ be the concentration of component $ B $, and $ T $ the temperature. We assume that the solidification of the alloy is governed by a phase equilibrium diagram consisting of a “liquidus” curve $ C = f_l (T) $ and a “solidus” curve $ C = f_s (T) $. We suppose that $ f_s $ and $ f_l $ are increasing functions in the variable $ T $ assuming that they verify the following inequalities:
\begin{equation*}
f_l (T_A) = f_s (T_A) <f_l (T) <f_s (T) <f_l (T_B) = f_s (T_B),
\end{equation*}
where $ T_A $ and $ T_B $ are the melting temperatures of $ A $ and $ B $ respectively. The material is in the solid phase if $ C> f_s (T) $ and in the liquid phase if $ C <f_l (T) $. When $ C $ is between $ f_s (T) $ and $ f_l (T) $ the state of the material is not well defined and it is known by mushy region according to the description of the model proposed in \cite{Rubi, SoWiAl, Wi:1978, WiSoAl:1982} which can by appreciated in the Figure \ref{fig1}.

The alloy is considered to be initially in a liquid state at constant temperature $ T_0 $ and constant concentration $ C_0 $. Then, a heat flow characterized by the constant $ q_0 $ is imposed on the fixed face $ x = 0 $ and a front of solidification $ x = s (t) $ begins instantly separating the alloy in solid state ($0< x < s(t) $) and liquid state ($ x> s(t) $). The mathematical formulation of this crystallization process consists of finding the temperature $ T = T (x, t) $ and the concentration $ C = C (x, t) $, both defined for $ x> 0 $ and $ t> 0 $, the free boundary $ x = s (t) $, defined for $ t> 0 $, and the critical solidification temperature $ T_k $ such that the following conditions are satisfied (problem $(P_1)$):
\begin{equation}\label{PS-AB}
 \begin{array}{llll}
i. & \al_s T_{s_{xx}}=T_{s_t} & 0<x<s(t),& t>0,\\
ii. & \al_l T_{l_{xx}}=T_{l_t} & s(t)<x,& t>0,\\
iii. & d_s C_{s_{xx}}=C_{s_t} & 0<x<s(t),& t>0,\\
iv. & d_l C_{l_{xx}}=C_{l_t} & s(t)<x,& t>0,\\
v. & k_s T_{s_x}(0,t)=\frac{q_0}{\sqrt{t}} & (q_0>0),& t>0,\\
vi. & T_s(s(t),t)= T_l(s(t),t)=T_k & & t>0,\\
vii. & T_l(x,0)=T_0 & T_A<T_0<T_B,& x>0,\\
viii. & T_l(\infty,t)=T_0 & T_A<T_0<T_B,& t>0,\\
ix. & C_{s_x}(0,t)= 0 & & t>0,\\
x. & C_s(s(t),t)= f_s(T_k) & & t>0,\\
xi. & C_l(x,0)=C_0 &  x>0, &\\
xii. & C_l(s(t),t)= f_l(T_k) & & t>0,\\
xiii. & k_s T_{s_x}(s(t),t)-k_l T_{l_x}(s(t),t)=\gamma\rho s'(t) & & t>0,\\
xiv. & d_l C_{l_x}(s(t),t)-d_s C_{s_x}(s(t),t)=\big[f_s(T_k)-f_l(T_k)\big] s'(t) & & t>0,\\
\end{array}
\end{equation}
where $ \rho, k, \al, d, \gamma $ represent the mass density, the thermal conductivity, the thermal diffusivity, the mass diffusion, and the latent heat of fusion, with $ s $ and $ l $ denoting the solid and the liquid phase respectively. When the condition on the fixed boundary $ x = 0 $ is given by a constant temperature $ T_A <T_s (0, t) = T_1 <T_0 $ (instead of condition (1.v)) the corresponding solidification problem was solved in \cite{Rubi, SoWiAl, WiSoAl:1982}.

A study of binary-alloy problems can be seen in \cite{Alexi:1985, CarMath, JaRaMa, PlaPleVan:2021, SunVo:1993, Tao:1980, White:1985, WiSoAl:1984}.
Recent works on the solidification of a binary alloy are \cite{AlMa:2006, AsVyMi:2021, CarMath, ChuLeeRoYoo, LeeAlexHuang:2012, MoSa:2020, PaMaWeKa:2020, Petrova:2010, PlaPleVan:2021, Sobolev:2015, Vo:2006, Vo:2008Bin, Vo:2008Num, ZiHetSlo:2017}.
The heat flux condition (1.v) imposed at the fixed face was first considered in \cite{Tar:1981} for a two-phase Stefan problem for a semi-infinite homogeneous material.
In \cite{BaTa:1985} the same heat flux condition was considered when a density jump is supposed for the two-phase Stefan problem.

We can also consider the problem $(P_2)$ which consists in finding the temperature $ T = T (x, t) $ and the concentration $ C = C (x, t) $, both defined for $ x> 0 $ and $ t> 0 $, the free boundary $ x = s (t) $, defined for $ t> 0 $, and the critical solidification temperature $ T_k $, such that the following conditions \eqref{PS-AB}(i)-(iv),(vi)-(xiv), and $(v')$ are verified, where
\begin{equation*}
v'. \quad k_s T_{s_x}(0,t)=\frac{h_0}{\sqrt{t}}(T_s(0,t)-T_1), \quad (h_0>0),\quad t>0.
\end{equation*}
The convective condition $(1. \, v')$ imposed at the fixed face was considered in \cite{Ta:2017} for a two-phase Stefan problem for a semi-infinite homogeneous material.

The goal of this paper is to find the necessary and / or sufficient conditions for data (initial and boundary conditions, and thermal coefficients of the binary alloy) in order to obtain an instantaneous phase-change process with the corresponding explicit solution of the similarity type. A review of explicit solution for Stefan-like problems is \cite{Tarzia}.

In Section \ref{simsol1} we obtain the necessary and sufficient condition \eqref{cotasq0} for problem $(P_1)$ in order to obtain the explicit solution \eqref{FBP1}-\eqref{TkyM}.
In Section \ref{equivprob1} we deduce the inequality \eqref{desigserflambda} for the coefficient $\mu$ that characterizes the Rubinstein free boundary $ x = s (t) $ for problem $(P_1)$ defined in \eqref{FBP1}.
In Section \ref{simsol2} we obtain the necessary and sufficient condition \eqref{cotash0} for problem $(P_2)$ in order to obtain the explicit solution \eqref{FBP1N}-\eqref{eqforlambdaN}.
Finally, in Section \ref{equivprob2} we deduce the inequality \eqref{desigserflambdaN} for the coefficient $\mu$ that characterizes the Rubinstein free boundary $ x = s (t) $ for problem $(P_2)$ defined in \eqref{FBP1N}.

\begin{figure}[t]
\centering
\includegraphics[width=.6\textwidth]{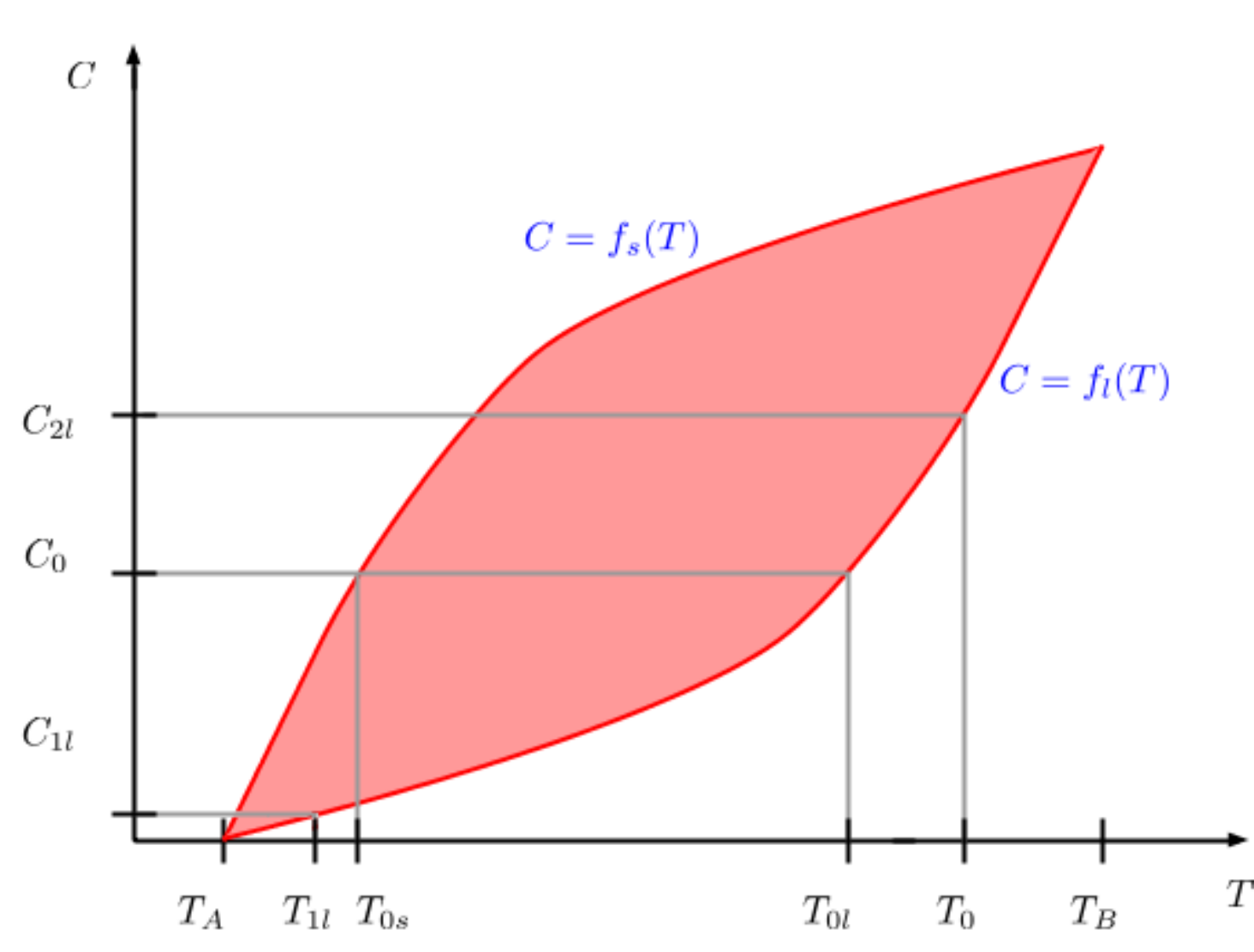}
\caption{Concentration vs. Temperature (phase equilibrium diagram with liquidus and solidus curves)}\label{fig1}
\end{figure}

\section{Explicit solution for the solidification of a binary alloy with a heat flux boundary condition}\label{simsol1}
In this Section we consider the problem $(P_1)$ defined by differential equations and conditions \eqref{PS-AB}(i)-(xiv).
\begin{theorem}\label{teoprincipal}
If $ q_0 $ verifies the following inequality
\begin{equation}\label{cotasq0}
\frac{(T_0-T_{0l})k_l}{\sqrt{\pi \al_l}}<q_0<\frac{(T_0-T_{0s})k_l}{\sqrt{\pi \al_l}}
\end{equation}
where $T_{0l}=f_l^{-1}(C_0)$ and $T_{0s}=f_s^{-1}(C_0)$, with $f_l^{-1}(C)=T$ is the inverse function of $f_l$ and $f_s^{-1}(C)=T$ is the inverse function of $f_s$ respectively, then there exists a unique solution of the similarity type for the free boundary problem $(P_1)$ which is given by:
\begin{equation}\label{FBP1}
s(t)=2\lambda \sqrt{\al_s t}, \quad t>0,
\end{equation}
\begin{equation}\label{FormTs}
T_s(x,t)=\bigg(T_k-\frac{q_0}{k_s} \sqrt{\pi\al_s}erf(\lambda)\bigg)+\frac{q_0}{k_s}\sqrt{\pi\al_s} erf\bigg(\frac{x}{2 \sqrt{\al_s t}}\bigg), \quad 0<x<s(t),\quad t>0,
\end{equation}
\begin{equation}\label{FormTl}
T_l(x,t)=T_0+\frac{T_k-T_0}{erfc\bigg(\frac{\sqrt{\al_s}}{\sqrt{\al_l}}\lambda\bigg)}+\frac{T_0-T_k}{erfc\bigg(\frac{\sqrt{\al_s}}{\sqrt{\al_l}}\lambda\bigg)}erf\bigg(\frac{x}{2 \sqrt{\al_l t}}\bigg), \quad s(t)<x,\quad t>0,
\end{equation}
\begin{equation}\label{FormCs}
C_s(x,t)=f_s(T_k), \quad 0<x<s(t),\quad t>0,
\end{equation}
\begin{equation}\label{FormCl}
C_l(x,t)=C_0+\frac{f_l(T_k)-C_0}{erfc\bigg(\frac{\sqrt{\al_s}}{\sqrt{d_l}}\lambda\bigg)}erfc\bigg(\frac{x}{2 \sqrt{d_l t}}\bigg), \quad s(t)<x,\quad t>0,
\end{equation}
where the unknowns $ T_k $ and $ \lambda $ (coefficient that characterizes the free boundary $ x = s (t) $) must satisfy the following equations:
\begin{equation}\label{TkyM}
T_k=F(\lambda),\quad M(\lambda)=\phi(T_k),\quad \lambda>0,\quad T_A<T_k<T_B,
\end{equation}
where the real functions $ F $, $ M $ and $ \phi $ are defined as follows:
\begin{equation*}
erf(x)=\frac{2}{\sqrt{\pi}}\int_0^x e^{-t^2}\dd t;\quad erfc(x)=1-erf(x),
\end{equation*}
\begin{equation}\label{defQF1F}
\begin{split}
Q(x)=\sqrt{\pi} x e^{x^2}erfc(x), \quad F_1(x)=erfc(x)e^{x^2},\quad x>0\\
\\
F(x)=T_0+\frac{\gamma \rho \al_l}{k_l}  Q\bigg(\frac{\sqrt{\al_s}}{\sqrt{\al_l}}x\bigg)-\frac{q_0}{k_l} \sqrt{\al_l \pi}e^{-x^2}F_1\bigg(\frac{\sqrt{\al_s}}{\sqrt{\al_l}}x\bigg),\quad x>0,
\end{split}
\end{equation}
\begin{equation*}
 M(x)=\bigg[Q\bigg(\frac{\sqrt{\al_s}}{\sqrt{d_l}}x\bigg)\bigg]^{-1},\quad x>0, \quad \quad
\phi(x)=\frac{f_s(x)-f_l(x)}{C_0-f_l(x)},\quad x\in (T_A,T_{0l})\cup(T_{0l},T_B).	
\end{equation*}
\end{theorem}
The proof of the theorem is based on the next proposition which will be done below.
\begin{proposition}\label{propsfcs}
The following properties are valid,
\begin{enumerate}[(a)]
\item $erf(x)$ is an strictly increasing function, with $erf(0^+)=0$ and $erf(+\infty)=1$.

\item $Q$ is an strictly increasing function, with $Q(0^+)=0$ and $Q(+\infty)=1$.

\item $F_1$ is an strictly decreasing function, with $F_1(0^+)=1$ and $F_1(+\infty)=0$.

\item $F$ is an strictly increasing function, with $F(0^+)=T_0-\frac{\sqrt{\pi \al_l}q_0}{k_l}$ and $ F(+\infty)=T_0+\frac{\gamma \rho \al_l}{k_l}$.

\item\label{propsfcs-e} $\phi$ is an strictly increasing function on $[T_{0s},T_{0l})$, with $\phi(T_{0s})=1$ and $\phi(T_{0l}^-)=\lim\limits_{x\rightarrow T_{0l}^-}\phi(x)=+\infty$.
\end{enumerate}
\end{proposition}

\begin{proof}
$(a)$, $(b)$, and $(c)$ follows from the definition of $erf$ function, \cite{Tar:1981} and \cite{Wi:1978}.
\begin{enumerate}
\item[(d)] Since $ Q $ is an increasing function and $ F_1 $ is a decreasing function, $ F $ is an increasing function. Further,
\begin{equation*}
F(0)=T_0+\frac{\gamma \rho \al_l}{k_l}  Q(0)-\frac{q_0}{k_l} \sqrt{\al_l \pi}e^{-0^2}F_1(0)=T_0-\frac{q_0}{k_l} \sqrt{\al_l \pi},
\end{equation*}
and
\begin{equation*}
F(+\infty)=\lim\limits_{x\rightarrow+\infty} T_0+\frac{\gamma \rho \al_l}{k_l}  Q\bigg(\frac{\sqrt{\al_s}}{\sqrt{\al_l}}x\bigg)-\frac{q_0}{k_l} \sqrt{\al_l \pi}e^{-x^2}F_1\bigg(\frac{\sqrt{\al_s}}{\sqrt{\al_l}}x\bigg)= T_0+\frac{\gamma \rho \al_l}{k_l},
\end{equation*}
\item[(e)] We have,
\begin{equation*}
\phi(x)=\frac{f_s(x)-C_0}{C_0-f_l(x)}+1,
\end{equation*}
and $ f_s (x) \geq C_0 $ for all $ x \geq T_{0s} $, $ f_l (x) <C_0 $ for all $ x <T_{0l} $. Consequently, since $ f_s $ and $ f_l $ are increasing functions, $ \phi $ is increasing on $ [T_{0s}, T_{0l}) $.
Moreover,
\begin{equation*}
\phi(T_{0s})=\frac{f_s(T_{0s})-f_l(T_{0s})}{C_0-f_l(T_{0s})}=1,
\end{equation*}
and since $ f_s (x)> f_l (x) $ for all $ x \in (T_A, T_B) $, in particular $ f_s (T_{0l})> f_l (T_{0l}) $, from where
\begin{equation*}
\phi(T_{0l}^-)=\lim\limits_{x\rightarrow T_{0l}^-}\frac{f_s(x)-f_l(x)}{C_0-f_l(x)}=+\infty.
\end{equation*}
\end{enumerate}
\vspace{-1cm}
\end{proof}
We now prove Theorem \ref{teoprincipal}.
\begin{proof}
The similarity solutions to the heat equation $ \al u_{xx} = u_t $ has the form $ u (x, t) = A + B erf \bigg (\frac{x}{2 \sqrt {\al t}} \bigg) $ with $A$ and $B$ real numbers to be determined.
Then, we can suppose that
\begin{equation}\label{FormSolExplic}
\begin{split}
T_s(x,t)=A_s^T+B_s^T erf\bigg(\frac{x}{2 \sqrt{\al_s t}}\bigg), \quad
T_l(x,t)=A_l^T+B_l^T erf\bigg(\frac{x}{2 \sqrt{\al_l t}}\bigg),\\
C_s(x,t)=A_s^C+B_s^C erf\bigg(\frac{x}{2 \sqrt{d_s t}}\bigg), \quad
C_l(x,t)=A_l^C+B_l^C erf\bigg(\frac{x}{2 \sqrt{d_l t}}\bigg).
\end{split}
\end{equation}
By \eqref{PS-AB}-$vi$ it results
\begin{equation*}
T_s(s(t),t)= A_s^T+B_s^T erf\bigg(\frac{s(t)}{2 \sqrt{\al_s t}}\bigg)=T_k,
\end{equation*}
where $T_k$ is a constant to be determined. Then it follows that $s(t)=2\lambda \sqrt{\al_s t}$ for some $\lambda >0$, and again by \eqref{PS-AB}-$vi$, we have
\begin{equation}\label{TsF}
T_s(s(t),t)= A_s^T+B_s^T erf(\lambda)=T_k,
\end{equation}
\begin{equation}\label{TlF}
T_l(s(t),t)= A_l^T+B_l^T erf\bigg(\frac{\sqrt{\al_s}}{\sqrt{\al_l}}\lambda\bigg)=T_k.
\end{equation}
On the other hand, from \eqref{PS-AB}-$v$, we obtain that
\begin{equation*}
k_s T_{s_x}(0,t)=\frac{k_s B_s^T}{\sqrt{\pi\al_s t}}=\frac{q_0}{\sqrt{t}},
\end{equation*}
from where
\begin{equation}\label{BsT}
B_s^T=\frac{q_0}{k_s}\sqrt{\pi\al_s }.
\end{equation}
Then, replacing \eqref{BsT} in \eqref{TsF}, we have
\begin{equation*}
A_s^T=T_k - \frac{q_0}{k_s}\sqrt{\pi\al_s } erf(\lambda).
\end{equation*}
From \eqref{PS-AB}-$vii$, it follows that
\begin{equation}\label{Tlt0}
T_l(x,0)=A_l^T+B_l^T=T_0.
\end{equation}
Subtracting \eqref{TlF} and \eqref{Tlt0}, we obtain that
\begin{equation*}
B_l^T erfc\bigg(\frac{\sqrt{\al_s}}{\sqrt{\al_l}}\lambda\bigg)=T_0-T_k,
\end{equation*}
or equivalently,
\begin{equation}\label{BlT}
B_l^T =\frac{T_0-T_k}{erfc\bigg(\frac{\sqrt{\al_s}}{\sqrt{\al_l}}\lambda\bigg)}.
\end{equation}
Replacing \eqref{BlT} in \eqref{Tlt0}, we conclude that
\begin{equation*}
A_l^T =T_0-\frac{T_0-T_k}{erfc\bigg(\frac{\sqrt{\al_s}}{\sqrt{\al_l}}\lambda\bigg)}.
\end{equation*}
Therefore, we obtain \eqref{FormTs} and \eqref{FormTl}.
Similarly, considering the equalities $ix$-$xii$ from \eqref{PS-AB}, we deduce that
\begin{equation*}
C_{s_x}(0,t)=\frac{B_s^C}{\sqrt{\pi\al_s t}}=0,
\end{equation*}
from where $B_s^C=0$, and
\begin{equation*}
C_s(s(t),t)= A_s^C+B_s^C erf\bigg(\frac{\sqrt{\al_s}}{\sqrt{d_s}}\lambda\bigg)=A_s^C = f_s(T_k).
\end{equation*}
Hence, $C_s\equiv  f_s(T_k)$.
On the other hand,
\begin{equation*}
C_l(x,0)=A_l^C+B_l^C =C_0,
\end{equation*}
and
\begin{equation*}
C_l(s(t),t)=A_l^C+B_l^C erf\bigg(\frac{\sqrt{\al_s}}{\sqrt{d_l}}\lambda\bigg)= f_l(T_k),
\end{equation*}
from where,
\begin{equation*}
A_l^C =C_0-\frac{C_0-f_l(T_k)}{erfc\bigg(\frac{\sqrt{\al_s}}{\sqrt{d_l}}\lambda\bigg)}, \quad
B_l^C=\frac{C_0-f_l(T_k)}{erfc\bigg(\frac{\sqrt{\al_s}}{\sqrt{d_l}}\lambda\bigg)}.
\end{equation*}
Therefore
\begin{equation*}
C_l(x,t)= C_0-\frac{C_0-f_l(T_k)}{erfc\bigg(\frac{\sqrt{\al_s}}{\sqrt{d_l}}\lambda\bigg)}+\frac{C_0-f_l(T_k)}{erfc\bigg(\frac{\sqrt{\al_s}}{\sqrt{d_l}}\lambda\bigg)} erf\bigg(\frac{x}{2 \sqrt{d_l t}}\bigg)=C_0+\frac{f_l(T_k)-C_0}{erfc\bigg(\frac{\sqrt{\al_s}}{\sqrt{d_l}}\lambda\bigg)}erfc\bigg(\frac{x}{2 \sqrt{d_l t}}\bigg),
\end{equation*}
that is \eqref{FormCl}.

From \eqref{PS-AB}-$xiii$, we have
\begin{equation*}
q_0 e^{-\lambda^2}+k_l \frac{T_k-T_0}{erfc\bigg(\frac{\sqrt{\al_s}}{\sqrt{\al_l}}\lambda\bigg)}\frac{e^{-\bigg(\frac{\sqrt{\al_s}}{\sqrt{\al_l}}\lambda\bigg)^2}}{\sqrt{\al_l \pi}}=\gamma\rho \lambda\sqrt{\al_s}.
\end{equation*}
Thus,
\begin{equation}\label{eq1}
\begin{split}
T_k&=T_0+(\gamma\rho \lambda\sqrt{\al_s}-q_0 e^{-\lambda^2})\frac{\sqrt{\al_l \pi}erfc\bigg(\frac{\sqrt{\al_s}}{\sqrt{\al_l}}\lambda\bigg)}{k_l e^{-\bigg(\frac{\sqrt{\al_s}}{\sqrt{\al_l}}\lambda\bigg)^2}}\\
&=T_0+\frac{\gamma\rho \al_l}{k_l }Q\bigg(\frac{\sqrt{\al_s}}{\sqrt{\al_l}}\lambda\bigg)-q_0 e^{-\lambda^2}\frac{\sqrt{\al_l \pi}}{k_l}F_1\bigg(\frac{\sqrt{\al_s}}{\sqrt{\al_l}}\lambda\bigg)\\
&=F(\lambda),
\end{split}
\end{equation}
where $Q$, $F_1$ y $F$ are given by \eqref{defQF1F}.

From \eqref{PS-AB}-$xiv$, we obtain
\begin{equation}\label{eq2}
\frac{1}{\sqrt{\pi}\frac{\sqrt{\al_s}}{\sqrt{d_l}}\lambda e^{\bigg(\frac{\sqrt{\al_s}}{\sqrt{d_l}}\lambda\bigg)^2}erfc\bigg(\frac{\sqrt{\al_s}}{\sqrt{d_l}}\lambda\bigg)}=\frac{f_s(T_k)-f_l(T_k)}{C_0-f_l(T_k)},
\end{equation}
that is
\begin{equation}
M(\lambda)=\phi(T_k).
\end{equation}
By Proposition \ref{propsfcs}, $Q$ and $F$ are increasing funtions. Then, $M$ is a decreasing funtion, and
\begin{equation}\label{valoresM}
M(0^+)=\lim\limits_{x\rightarrow 0^+}\frac{1}{Q(x)}=+\infty, \quad M(+\infty)=\frac{1}{Q(+\infty)}=1,
\end{equation}
\begin{equation}\label{valoresF}
F(0^+)=T_0-\frac{\sqrt{\pi \al_l}q_0}{k_l},\quad F(+\infty)=T_0+\frac{\gamma \rho \al_l}{k_l}.
\end{equation}
By \eqref{eq1} y \eqref{eq2}, we have
\begin{equation}\label{eqlamb}
M(\lambda)=\phi(F(\lambda)).
\end{equation}
Then, taking into account \eqref{valoresM}, \eqref{valoresF} and of Proposition \ref{propsfcs}-5, we can assure the existence and uniqueness of $\lambda$ verifying \eqref{eqlamb}, if $ T_{0s} \leq F (0) <F (+ \infty) <T_{0l} $, or equivalently \eqref{cotasq0}.
\end{proof}
\begin{remark}\label{obs}
Define $T_1 = T_s(0,t)$. Then,
\begin{equation*}
T_1=T_k-\frac{q_0}{k_s}\sqrt{\al_s \pi} erf(\lambda), 
\end{equation*}
which is a constant independent of time $t$. Moreover,
\begin{equation*}
T_1\leq T_s(x,t)<T_k, \quad 0<x<s(t), \quad t>0. 
\end{equation*}
In fact, since $ erf (x) $ is an increasing function, and for $x\in (0,s(t))$, $0< erf\bigg(\frac{x}{2\sqrt{\al_s t}}\bigg)<erf(\lambda)$, we have
\begin{equation*}
T_1=T_s(0,t)< T_s(x,t)< T_s(s(t),t)=T_k, \quad 0<x<s(t), \quad t>0. 
\end{equation*}
\end{remark}
\begin{corollary}
Define $C_{1l}=f_l(T_1)$ and $C_{2l}=f_l(T_0)$. Then, $C_0\in [C_{1l},C_{2l}]$.
\end{corollary}
\begin{proof}
By definition, $C_0=f_l(T_{0l})$. By hypothesis, $f_l(T_l)\geq C_l$. Hence
\begin{equation}\label{cotaC2l}
C_{2l}=f_l(T_0)=f_l(T_l (x,0))\geq C_l (x,0)=C_0.
\end{equation}
On the other hand, there are values $ x_0 $ and $ t_0 $, with $ 0 <x_0 <s (t_0) $, such that
\begin{equation*}
T_{0s}=f_s^{-1}(C_0)=T_s(x_0,t_0),
\end{equation*}
and since $ f_l (T) <f_s (T) $, it results $T_{0s}=f_s^{-1}(C_0)<f_l^{-1}(C_0)=T_{0l}$.

Thus, by Remark \ref{obs}, it results $T_1\leq T_s(x_0,t_0)=T_{0s}\leq T_{0l}$, from where, taking into account that $ f_l $ is an increasing function, we obtain that
\begin{equation}\label{cotaC1l}
C_{1l}=f_l(T_1)\leq f_l (T_{0l})=C_0.
\end{equation}
By \eqref{cotaC2l} and \eqref{cotaC1l} we conclude that $C_0\in [C_{1l},C_{2l}]$.
\end{proof}

\section{An Inequality for the Coefficient $\mu$ of the Rubinstein Free Boundary}\label{equivprob1}
For the solution given in Theorem \ref{teoprincipal}, the temperature on the fixed face $x=0$ is given by
\begin{equation}\label{temp0}
\widetilde{T}_1 = T_s(0,t)=T_k-\frac{q_0}{k_s}\sqrt{\al_s \pi} erf(\lambda).
\end{equation}
Since $\widetilde{T}_1<T_k$, we can consider the problem $(P_R)$ given by \eqref{PS-AB}$(i)$-$(iv)$,$(\widetilde{v})$,$(vi)$-$(xiv)$, where
\begin{equation}\label{convcond}
\begin{array}{lll}
\widetilde{v}. & \widetilde{T}_s(0,t)=T_1, & t>0.
\end{array}
\end{equation}
Problem $(P_R)$ has a unique solution, known as Rubinstein solution, given by \cite{Rubi}, that is
\begin{equation}\label{RubSol}
\widetilde{T}_s(x,t)=\mathscr{A}_s^T+\mathscr{B}_s^T erf\bigg(\frac{x}{2 \sqrt{\al_s t}}\bigg), \quad
\widetilde{T}_l(x,t)=\mathscr{A}_l^T+\mathscr{B}_l^T erf\bigg(\frac{x}{2 \sqrt{\al_l t}}\bigg),
\end{equation}
\begin{equation}
\widetilde{C}_s(x,t)=\mathscr{A}_s^C+\mathscr{B}_s^C erf\bigg(\frac{x}{2 \sqrt{d_s t}}\bigg), \quad
\widetilde{C}_l(x,t)=\mathscr{A}_l^C+\mathscr{B}_l^C erf\bigg(\frac{x}{2 \sqrt{d_l t}}\bigg),
\end{equation}
\begin{equation}
\widetilde{s}(t)=2\mu \sqrt{\al_s t}
\end{equation}
where
\begin{equation}
\mathscr{A}_s^T(\mu)=T_1,
\quad\quad
\mathscr{B}_s^T(\mu)=\frac{\widetilde{T}_k-T_1}{erf(\mu)},
\end{equation}
\begin{equation}
\mathscr{A}_l^T(\mu)=T_0+\frac{\widetilde{T}_k-T_0}{erfc\bigg(\frac{\sqrt{\al_s}}{\sqrt{\al_l}}\mu\bigg)},
\quad\quad
\mathscr{B}_l^T(\mu)=\frac{T_0-\widetilde{T}_k}{erfc\bigg(\frac{\sqrt{\al_s}}{\sqrt{\al_l}}\mu\bigg)},
\end{equation}
\begin{equation}
\mathscr{A}_s^C(\mu)=f_s(\widetilde{T}_k),
\quad\quad
\mathscr{B}_s^C(\mu)=0,
\end{equation}
\begin{equation}
\mathscr{A}_l^C(\mu)=C_0+\frac{f_l(\widetilde{T}_k)-C_0}{erfc\bigg(\frac{\sqrt{\al_s}}{\sqrt{d_l}}\mu\bigg)},
\quad\quad
\mathscr{B}_l^C(\mu)=\frac{C_0-f_l(\widetilde{T}_k)}{erfc\bigg(\frac{\sqrt{\al_s}}{\sqrt{d_l}}\mu\bigg)},
\end{equation}
and the unknowns $ \widetilde{T}_k $ and $ \mu $ (coefficient that characterizes the free boundary $ x = s (t) $) satisfy the following equations:
\begin{equation}\label{condTkN}
\widetilde{T}_k=G(\mu),\quad M(\mu)=\phi(\widetilde{T}_k),\quad \mu>0,\quad T_A<\widetilde{T}_k<T_B,
\end{equation}
where $G$ is defined as follows:
\begin{equation*}
G(x)=\frac{\gamma \rho\al_s\al_l Q_1(x)Q\bigg(\frac{\sqrt{\al_s}}{\sqrt{\al_l}}x\bigg)+T_1 k_s \al_l Q\bigg(\frac{\sqrt{\al_s}}{\sqrt{\al_l}}x\bigg) + T_0 k_l \al_s Q_1(x)}{k_s \al_l Q\bigg(\frac{\sqrt{\al_s}}{\sqrt{\al_l}}x\bigg) + k_l \al_s Q_1(x)},\quad x>0,
\end{equation*}
\begin{equation}\label{defQ1}
Q_1(x)=\sqrt{\pi} x e^{x^2}erf(x).
\end{equation}
\begin{theorem}\label{equivtheo}
If in the Rubinstein solution \eqref{RubSol}-\eqref{condTkN} we take $T_1=\widetilde{T}_1$, under condition \eqref{cotasq0} and \eqref{temp0}, we obtain that the two free boundaries problems $(P_1)$ and $(P_R)$ are equivalents, that is
\begin{equation}\label{equiv1}
\mu=\lambda, \quad \widetilde{T}_k=T_k,
\end{equation}
\begin{equation}\label{TyCtildes}
T_s(x,t)=\widetilde{T}_s(x,t), \quad T_l(x,t)=\widetilde{T}_l(x,t), \quad C_s(x,t)=\widetilde{C}_s(x,t), \quad C_l(x,t)=\widetilde{C}_l(x,t).
\end{equation}
\end{theorem}
\begin{proof}
By \eqref{TkyM} we have that $M(\lambda)=\phi(T_k)$. Taking into account that $T_1=\widetilde{T}_1$ is given by \eqref{temp0}, we have
\begin{equation}\label{cuentasG}
\begin{split}
G(\lambda)&=\frac{\gamma \rho\al_s\al_l Q_1(\lambda)Q\bigg(\frac{\sqrt{\al_s}}{\sqrt{\al_l}}\lambda\bigg)+\widetilde{T}_1 k_s \al_l Q\bigg(\frac{\sqrt{\al_s}}{\sqrt{\al_l}}\lambda\bigg) + T_0 k_l \al_s Q_1(\lambda)}{k_s \al_l Q\bigg(\frac{\sqrt{\al_s}}{\sqrt{\al_l}}\lambda\bigg) + k_l \al_s Q_1(\lambda)}\\
&=\frac{\gamma \rho\al_s\al_l Q_1(\lambda)Q\bigg(\frac{\sqrt{\al_s}}{\sqrt{\al_l}}\lambda\bigg)+T_k k_s \al_l Q\bigg(\frac{\sqrt{\al_s}}{\sqrt{\al_l}}\lambda\bigg)-q_0 \sqrt{\al_s \pi} erf(\lambda)\al_l Q\bigg(\frac{\sqrt{\al_s}}{\sqrt{\al_l}}\lambda\bigg) + T_0 k_l \al_s Q_1(\lambda)}{k_s \al_l Q\bigg(\frac{\sqrt{\al_s}}{\sqrt{\al_l}}\lambda\bigg) + k_l \al_s Q_1(\lambda)}\\
&=\frac{\al_s k_l Q_1(\lambda)\bigg(T_0+\frac{\gamma \rho\al_l}{k_l} Q\bigg(\frac{\sqrt{\al_s}}{\sqrt{\al_l}}\lambda\bigg) - \frac{q_0 \al_l}{\sqrt{\al_s}k_l} \frac{e^{-\lambda^2}}{\lambda}Q\bigg(\frac{\sqrt{\al_s}}{\sqrt{\al_l}}\lambda\bigg) \bigg)
+T_k k_s \al_l Q\bigg(\frac{\sqrt{\al_s}}{\sqrt{\al_l}}\lambda\bigg)}{k_s \al_l Q\bigg(\frac{\sqrt{\al_s}}{\sqrt{\al_l}}\lambda\bigg) + k_l \al_s Q_1(\lambda)}\\
&=\frac{\al_s k_l Q_1(\lambda)\bigg(T_0+\frac{\gamma \rho\al_l}{k_l} Q\bigg(\frac{\sqrt{\al_s}}{\sqrt{\al_l}}\lambda\bigg) - \frac{q_0 }{k_l}\sqrt{\al_l \pi} e^{-\lambda^2}F_1\bigg(\frac{\sqrt{\al_s}}{\sqrt{\al_l}}\lambda\bigg) \bigg)
+T_k k_s \al_l Q\bigg(\frac{\sqrt{\al_s}}{\sqrt{\al_l}}\lambda\bigg)}{k_s \al_l Q\bigg(\frac{\sqrt{\al_s}}{\sqrt{\al_l}}\lambda\bigg) + k_l \al_s Q_1(\lambda)}\\
&=\frac{\al_s k_l Q_1(\lambda)F(\lambda)
+T_k k_s \al_l Q\bigg(\frac{\sqrt{\al_s}}{\sqrt{\al_l}}\lambda\bigg)}{k_s \al_l Q\bigg(\frac{\sqrt{\al_s}}{\sqrt{\al_l}}\lambda\bigg) + k_l \al_s Q_1(\lambda)}\\
&=\frac{\al_s k_l Q_1(\lambda)T_k
+T_k k_s \al_l Q\bigg(\frac{\sqrt{\al_s}}{\sqrt{\al_l}}\lambda\bigg)}{k_s \al_l Q\bigg(\frac{\sqrt{\al_s}}{\sqrt{\al_l}}\lambda\bigg) + k_l \al_s Q_1(\lambda)}\\
&=T_k.
\end{split}
\end{equation}
Then,
\begin{equation*}
\begin{cases}
M(\lambda)=\phi(T_k),\\
G(\lambda)=T_k,
\end{cases}
\end{equation*}
and by the uniqueness of the pair $\widetilde{T}_k,\mu$ in \eqref{condTkN}, we have $\lambda=\mu$, and $T_k=\widetilde{T}_k$.

On the other hand, taking into account that $\lambda=\mu$, we have
\begin{equation*}
\begin{split}
A_s^T(x)&=\mathscr{A}_s^T(x), \quad \forall x>0, \quad\quad\quad A_l^T(x)=\mathscr{A}_l^T(x), \quad \forall x>0,\\
A_s^C(x)&=\mathscr{A}_s^C(x), \quad \forall x>0, \quad\quad\quad A_l^C(x)=\mathscr{A}_l^C(x), \quad \forall x>0,\\
B_s^T(x)&=\mathscr{B}_s^T(x), \quad \forall x>0, \quad\quad\quad B_l^T(x)=\mathscr{B}_l^T(x), \quad \forall x>0,\\
B_s^C(x)&=\mathscr{B}_s^C(x), \quad \forall x>0, \quad\quad\quad B_l^C(x)=\mathscr{B}_l^C(x), \quad \forall x>0,
\end{split}
\end{equation*}
that is \eqref{TyCtildes}.
\end{proof}
\begin{corollary}
Under the hypothesis of Theorem \ref{equivtheo}, the coefficient $\mu$ of the free boundary $\widetilde{s}(t)$ and the initial temperature $\widetilde{T}_k$ of solidification, corresponding to the Rubinstein solution to the free boundary problem $(P_R)$ satisfies the inequalities
\begin{equation}\label{desigserflambda}
\frac{\sqrt{ \al_l}}{\sqrt{\al_s}}\frac{\widetilde{T}_k-T_1}{T_0-T_{0s}}\frac{k_s}{k_l}<erf(\mu)<\frac{\sqrt{ \al_l}}{\sqrt{\al_s}}\frac{\widetilde{T}_k-T_1}{T_0-T_{0l}}\frac{k_s}{k_l}.
\end{equation}
\end{corollary}
\begin{proof}
From the solution to problem $P_R$, equivalent to problem $P_1$, by using condition \eqref{PS-AB}-v, we have \eqref{temp0}, from where,
\begin{equation*}
q_0=\frac{k_s(\widetilde{T}_k-T_1)}{\sqrt{\al_s \pi} erf(\mu)},
\end{equation*}
where $T_1$ is the temperature boundary condition for the Rubinstein solution.

By the equivalence between problems $(P_1)$ and $(P_R)$, $q_0$ must satisfy condition \eqref{cotasq0}.
Then, the next inequalities hold:
\begin{equation*}
\frac{(T_0-T_{0l})k_l}{\sqrt{\pi \al_l}}<\frac{k_s(T_k-T_1)}{\sqrt{\al_s \pi} erf(\mu)}<\frac{(T_0-T_{0s})k_l}{\sqrt{\pi \al_l}},
\end{equation*}
or equivalently \eqref{desigserflambda}.
\end{proof}
\begin{remark}\label{Remark2}
The inequalities \eqref{desigserflambda} has a complete physical meaning for the solution to problem $P_R$ when the corresponding parameters verifies the inequality
\begin{equation*}
\frac{\sqrt{ \al_l}}{\sqrt{\al_s}}\frac{\widetilde{T}_k-T_1}{T_0-T_{0l}}\frac{k_s}{k_l}<1.
\end{equation*}
\end{remark}
\begin{remark}
The results in this section generalizes \cite{Tar:1981}. In fact, if $C(x,t)$ is constant and if we consider only the thermal problem, then condition \eqref{cotasq0} is the inequality given in \cite[Lemma 1]{Tar:1981} obtained for the two-phase Stefan problem with a heat flux condition at $x=0$.
\end{remark}

\section{Explicit solution for the binary alloy solidification problem with a convective boundary condition}\label{simsol2}
In this Section we consider the next two-phase Rubinstein type binary-alloy solidification problem defined by (Problem $(P_2)$):
\begin{equation}\label{PS-ABalt}
 \begin{array}{llll}
i. & \al_s T_{s_{xx}}=T_{s_t} & 0<x<s(t),& t>0,\\
ii. & \al_l T_{l_{xx}}=T_{l_t} & s(t)<x,& t>0,\\
iii. & d_s C_{s_{xx}}=C_{s_t} & 0<x<s(t),& t>0,\\
iv. & d_l C_{l_{xx}}=C_{l_t} & s(t)<x,& t>0,\\
v. & k_s T_{s_x}(0,t)=\frac{h_0}{\sqrt{t}}\bigg(T_s(0,t)-T_\infty\bigg) & (h_0>0),& t>0,\\
vi. & T_s(s(t),t)= T_l(s(t),t)=T_k & & t>0,\\
vii. & T_l(x,0)=T_0 & T_A<T_0<T_B,& x>0,\\
viii. & T_l(\infty,t)=T_0 & T_A<T_0<T_B,& t>0,\\
ix. & C_{s_x}(0,t)= 0 & & t>0,\\
x. & C_s(s(t),t)= f_s(T_k) & & t>0,\\
xi. & C_l(x,0)=C_0 &  x>0, &\\
xii. & C_l(s(t),t)= f_l(T_k) & & t>0,\\
xiii. & k_s T_{s_x}(s(t),t)-k_l T_{l_x}(s(t),t)=\gamma\rho s'(t) & & t>0,\\
xiv. & d_l C_{l_x}(s(t),t)-d_s C_{s_x}(s(t),t)=\bigg[f_s(T_k)-f_l(T_k)\bigg] s'(t) & & t>0,\\
\end{array}
\end{equation}
where $T_\infty$ is the bulk temperature at a large distance from the fixed face $x=0$.

Following the study done in Section \ref{simsol1}, we can obtain the next result.
\begin{theorem}\label{teoprincipal2}
If the coefficient $ h_0 $ which characterized the convective boundary condition \eqref{PS-ABalt}-$(v)$ verifies the following inequalities
\begin{equation}\label{cotash0}
\frac{(T_0-T_{0l})k_l}{(T_{0l}-T_\infty)\sqrt{\pi \al_l}}<h_0<\frac{(T_0-T_{0s})k_l}{(T_{0s}-T_\infty)\sqrt{\pi \al_l}},
\end{equation}
where $T_{0l}=f_l^{-1}(C_0)$ and $T_{0s}=f_s^{-1}(C_0)$, where $f_l^{-1}(C)=T$ is the inverse function of $f_l$ and $f_s^{-1}(C)=T$ is the inverse function of $f_s$ respectively, the free boundary problem \eqref{PS-ABalt} has a unique similarity type solution given by:
\begin{equation}\label{FBP1N}
\widehat{s}(t)=2\delta \sqrt{\al_s t}, \quad t>0,
\end{equation}
\begin{equation}\label{FormTsN}
\widehat{T}_s(x,t)=\widehat{T}_k+\frac{h_0 \sqrt{\pi\al_s}(\widehat{T}_k - T_\infty)}{k_s+h_0 \sqrt{\pi\al_s}erf(\delta)}\bigg( erf\bigg(\frac{x}{2\sqrt{\al_s t}}\bigg)-erf(\delta)\bigg), \quad 0<x<\widehat{s}(t),\quad t>0,
\end{equation}
\begin{equation}\label{FormTlN}
\widehat{T}_l(x,t)=T_0+\frac{\widehat{T}_k-T_0}{erfc\bigg(\frac{\sqrt{\al_s}}{\sqrt{\al_l}}\delta\bigg)}erfc\bigg(\frac{x}{2 \sqrt{\al_l t}}\bigg), \quad \widehat{s}(t)<x,\quad t>0,
\end{equation}
\begin{equation}\label{FormCsN}
\widehat{C}_s(x,t)=f_s(\widehat{T}_k), \quad 0<x<\widehat{s}(t),\quad t>0,
\end{equation}
\begin{equation}\label{FormClN}
\widehat{C}_l(x,t)=C_0+\frac{f_l(\widehat{T}_k)-C_0}{erfc\bigg(\frac{\sqrt{\al_s}}{\sqrt{d_l}}\delta\bigg)}erfc\bigg(\frac{x}{2 \sqrt{d_l t}}\bigg), \quad \widehat{s}(t)<x,\quad t>0,
\end{equation}
where the unknowns $ \widehat{T}_k $ and $ \delta $ (coefficient that characterizes the free boundary $ x = s (t) $) satisfy the following equations:
\begin{equation}\label{eqforlambdaN}
\widehat{T}_k=W(\delta),\quad M(\delta)=\phi(\widehat{T}_k),\quad \delta>0,\quad T_A<\widehat{T}_k<T_B,
\end{equation}
where the real function $ W $, is defined as follows:
\begin{equation*}
 W(x)=T_\infty+\frac{T_0 -T_\infty}{F_2(x)+1}+\frac{\gamma\rho \sqrt{\pi\al_s \al_l}}{H(x)}
\end{equation*}
\begin{equation*}
F_2(x)=\frac{h_0 k_s \sqrt{\pi \al_l} e^{-x^2}F_1\bigg(\frac{\sqrt{\al_s}}{\sqrt{\al_l}}x\bigg)}{k_l (k_s+h_0 \sqrt{\pi\al_s}erf(x))},
\end{equation*}
\begin{equation*}
H(x)= \frac{h_0 k_s \sqrt{\pi \al_l}}{x e^{x^2}(k_s+h_0 \sqrt{\pi\al_s}erf(x))}+\frac{k_l\sqrt{\pi \al_s}}{\sqrt{\al_l}Q\bigg(\frac{\sqrt{\al_s}}{\sqrt{\al_l}}x\bigg)}.
\end{equation*}
\end{theorem}
The proof of the theorem is based on the next proposition.
\begin{proposition}\label{propsfcsN}
The following properties are valid,
\begin{enumerate}[(a)]
\item\label{propG} $F_2$ is an strictly decreasing function, with $F_2(0^+)=\frac{h_0}{k_l}\sqrt{\al_l\pi}$ and $F_2(+\infty)=0$.

\item\label{propH} $H$ is an strictly decreasing function, with $H(0^+)=+\infty$ and $H(+\infty)=\frac{\sqrt{\pi\al_s}}{\sqrt{\al_l}}k_l$.

\item\label{propW} $W$ is an strictly increasing function, with $W(0^+)=T_\infty+\frac{T_0-T_\infty}{1+\frac{h_0\sqrt{\pi \al_l}}{k_l}}$ and $ W(+\infty)=T_0+\frac{\gamma \rho \al_l}{k_l}$.
\end{enumerate}
\end{proposition}
\begin{proof}
We obtain $(a)$ and $(b)$ by Proposition \ref{propsfcs}, and $(c)$ by using properties $(a)$ and $(b)$.
\end{proof}
We now prove Theorem \ref{teoprincipal2}.
\begin{proof}
The similarity solutions to the heat equation $ \al u_{xx} = u_t $ has now the form $ u (x, t) = D + E erf \bigg (\frac{x}{2 \sqrt {\al t}} \bigg) $, where the real coefficients $D$ and $E$ must be determined.
Then, we can write
\begin{equation}
\begin{split}
\widehat{T}_s(x,t)=D_s^T+E_s^T erf\bigg(\frac{x}{2 \sqrt{\al_s t}}\bigg), \quad
\widehat{T}_l(x,t)=D_l^T+E_l^T erf\bigg(\frac{x}{2 \sqrt{\al_l t}}\bigg),\\
\widehat{C}_s(x,t)=D_s^C+E_s^C erf\bigg(\frac{x}{2 \sqrt{d_s t}}\bigg), \quad
\widehat{C}_l(x,t)=D_l^C+E_l^C erf\bigg(\frac{x}{2 \sqrt{d_l t}}\bigg).
\end{split}
\end{equation}
By \eqref{PS-ABalt}-$vi$ it results
\begin{equation*}
\widehat{T}_s(s(t),t)= D_s^T+E_s^T erf\bigg(\frac{s(t)}{2 \sqrt{\al_s t}}\bigg)=\widehat{T}_k,
\end{equation*}
where $\widehat{T}_k$ is a constant to be determined. Then it follows that $s(t)=2\delta \sqrt{\al_s t}$ for some $\delta >0$, and again by \eqref{PS-AB}-$vi$, we get
\begin{equation}\label{TsF2}
\widehat{T}_s(s(t),t)= D_s^T+E_s^T erf(\delta)=\widehat{T}_k,
\end{equation}
\begin{equation}\label{TlF2}
\widehat{T}_l(s(t),t)= D_l^T+E_l^T erf\bigg(\frac{\sqrt{\al_s}}{\sqrt{\al_l}}\delta\bigg)=\widehat{T}_k.
\end{equation}
On the other hand,
\begin{equation*}
\widehat{T}_s(0,t)=D_s^T,
\end{equation*}
and then, from \eqref{PS-ABalt}-$v$, we obtain that
\begin{equation*}
k_s T_{s_x}(0,t)=k_s E_s^T \frac{2}{\sqrt{\pi}}\frac{1}{2 \sqrt{\al_s t}}=\frac{k_s E_s^T}{\sqrt{\pi\al_s t}}=\frac{h_0}{\sqrt{t}}\bigg(D_s^T-T_\infty\bigg) .
\end{equation*}
that is
\begin{equation}\label{DsT}
D_s^T=T_\infty+\frac{k_s E_s^T}{h_0\sqrt{\pi\al_s}}.
\end{equation}
Then, replacing \eqref{DsT} in \eqref{TsF2}, we have
\begin{equation*}\label{formEst}
E_s^T=\frac{h_0\sqrt{\pi\al_s}(\widehat{T}_k-T_\infty)}{k_s+h_0\sqrt{\pi\al_s}erf(\delta) },
\end{equation*}
and then we have
\begin{equation*}\label{formDst}
D_s^T=T_\infty+\frac{k_s(\widehat{T}_k-T_\infty)}{k_s+h_0\sqrt{\pi\al_s}erf(\delta) }.
\end{equation*}
From \eqref{PS-ABalt}-$vii$, it follows that
\begin{equation}\label{Tlt02}
T_0=\widehat{T}_l(x,0)=D_l^T+E_l^T.
\end{equation}
Subtracting \eqref{TlF2} and \eqref{Tlt02}, we obtain that
\begin{equation}\label{ElT}
E_l^T =\frac{T_0-\widehat{T}_k}{erfc\bigg(\frac{\sqrt{\al_s}}{\sqrt{\al_l}}\delta\bigg)}.
\end{equation}
Replacing \eqref{ElT} in \eqref{Tlt02}, we conclude that
\begin{equation*}
D_l^T =T_0+\frac{\widehat{T}_k-T_0}{erfc\bigg(\frac{\sqrt{\al_s}}{\sqrt{\al_l}}\delta\bigg)}.
\end{equation*}
Therefore,
\begin{equation*}
\widehat{T}_s(x,t)=\widehat{T}_k+\frac{h_0 \sqrt{\pi\al_s}(\widehat{T}_k - T_\infty)}{k_s+h_0 \sqrt{\pi\al_s}erf(\delta)}\bigg( erf\bigg(\frac{x}{2\sqrt{\al_s t}}\bigg)-erf(\delta)\bigg),
\end{equation*}
and
\begin{equation*}
\widehat{T}_l(x,t)=T_0+\frac{\widehat{T}_k-T_0}{erfc\bigg(\frac{\sqrt{\al_s}}{\sqrt{\al_l}}\delta\bigg)}erfc\bigg(\frac{x}{2 \sqrt{\al_l t}}\bigg),
\end{equation*}
that is, \eqref{FormTsN} and \eqref{FormTlN}.

Similarly, considering the equalities $ix$-$xii$ from \eqref{PS-AB}, we deduce that
\begin{equation*}
D_s^C = f_s(\widehat{T}_k),\quad\quad\quad E_s^C=0,
\end{equation*}
\begin{equation*}
D_l^C =C_0+\frac{f_l(\widehat{T}_k)-C_0}{erfc\bigg(\frac{\sqrt{\al_s}}{\sqrt{d_l}}\delta\bigg)},\quad\quad\quad E_l^C=\frac{C_0-f_l(\widehat{T}_k)}{erfc\bigg(\frac{\sqrt{\al_s}}{\sqrt{d_l}}\delta\bigg)},
\end{equation*}
that is \eqref{FormCsN} and \eqref{FormClN}.

From \eqref{PS-ABalt}-$xiii$, we have
\begin{equation}
 k_s\frac{h_0 (\widehat{T}_k - T_\infty)}{k_s+h_0 \sqrt{\pi\al_s}erf(\delta)}e^{-\delta^2}+k_l \frac{\widehat{T}_k-T_0}{erfc\bigg(\frac{\sqrt{\al_s}}{\sqrt{\al_l}}\delta\bigg)}\frac{1}{\sqrt{\pi \al_l}}e^{-\bigg(\frac{\sqrt{\al_s}}{\al_l}\delta\bigg)^2}=\gamma\rho \delta\sqrt{\al_s},
\end{equation}
from where,
\begin{equation}\label{TkW}
\begin{split}
\widehat{T}_k&=\frac{\gamma\rho \sqrt{\pi\al_s \al_l}(k_s+h_0 \sqrt{\pi\al_s }erf(\delta))\delta erfc\bigg(\frac{\sqrt{\al_s}}{\sqrt{\al_l}}\delta\bigg)}{ h_0 k_s \sqrt{\pi \al_l}e^{-\delta^2}erfc\bigg(\frac{\sqrt{\al_s}}{\sqrt{\al_l}}\delta\bigg)+k_l e^{-\bigg(\frac{\sqrt{\al_s}}{\sqrt{\al_l}}\delta\bigg)^2}(k_s+h_0 \sqrt{\pi\al_s}erf(\delta))}\\
&\quad +\frac{h_0 k_s T_\infty \sqrt{\pi \al_l }e^{-\delta^2}erfc\bigg(\frac{\sqrt{\al_s}}{\sqrt{\al_l}}\delta\bigg)+k_l T_0 (k_s+h_0 \sqrt{\pi\al_s }erf(\delta))e^{-\bigg(\frac{\sqrt{\al_s}}{\al_l}\delta\bigg)^2}}{ h_0 k_s \sqrt{\pi \al_l} e^{-\delta^2} erfc\bigg(\frac{\sqrt{\al_s}}{\sqrt{\al_l}}\delta\bigg)+k_l (k_s+h_0 \sqrt{\pi\al_s}erf(\delta))e^{-\bigg(\frac{\sqrt{\al_s}}{\al_l}\delta\bigg)^2}}\\
&=\frac{\gamma\rho \sqrt{\pi\al_s \al_l}(k_s+h_0 \sqrt{\pi\al_s }erf(\delta))\delta }{ h_0 k_s \sqrt{\pi \al_l}e^{-\delta^2}+k_l \frac{e^{-\bigg(\frac{\sqrt{\al_s}}{\sqrt{\al_l}}\delta\bigg)^2}}{erfc\bigg(\frac{\sqrt{\al_s}}{\sqrt{\al_l}}\delta\bigg)}(k_s+h_0 \sqrt{\pi\al_s}erf(\delta))}\\
&\quad +\frac{\frac{h_0 k_s  \sqrt{\pi \al_l }e^{-\delta^2}e^{\bigg(\frac{\sqrt{\al_s}}{\al_l}\delta\bigg)^2} erfc\bigg(\frac{\sqrt{\al_s}}{\sqrt{\al_l}}\delta\bigg)}{k_l (k_s+h_0 \sqrt{\pi\al_s}erf(\delta))}T_\infty+T_0}{\frac{h_0 k_s \sqrt{\pi \al_l} e^{-\delta^2}e^{\bigg(\frac{\sqrt{\al_s}}{\al_l}\delta\bigg)^2} erfc\bigg(\frac{\sqrt{\al_s}}{\sqrt{\al_l}}\delta\bigg)}{k_l (k_s+h_0 \sqrt{\pi\al_s}erf(\delta))}+1}\\
&=\frac{\gamma\rho \sqrt{\pi\al_s \al_l}}{H(\delta)}+T_\infty+\frac{T_0 -T_\infty}{F_2(\delta)+1}\\
&=W(\delta).
\end{split}
\end{equation}
From \eqref{PS-ABalt}-$xiv$, we obtain
\begin{equation*}
\frac{C_0-f_l(\widehat{T}_k)}{erfc\bigg(\frac{\sqrt{\al_s}}{\sqrt{d_l}}\delta\bigg)}\frac{\sqrt{ d_l }}{\sqrt{\pi}}e^{-\bigg(\frac{\sqrt{\al_s}}{\sqrt{d_l}\delta}\bigg)^2}=\big[f_s(\widehat{T}_k)-f_l(\widehat{T}_k)\big] \delta\sqrt{\al_s}
\end{equation*}
that is
\begin{equation}\label{eq2N}
\frac{1}{Q\bigg(\frac{\sqrt{\al_s}}{\sqrt{d_l}}\delta\bigg)}=\phi(\widehat{T}_k),
\end{equation}
and then \eqref{eqforlambdaN} holds.

By \eqref{TkW} and \eqref{eq2N}, we have
\begin{equation*}
M(\delta)=\phi(W(\delta)).
\end{equation*}
We know that $M$ is a strictly decreasing function with, $M(0^+)=+\infty, M(+\infty)=1$.
On the other hand, by Proposition \ref{propsfcsN}-\ref{propW}, $W$ is an strictly increasing function.
Then, taking into account Proposition \ref{propsfcs}-\ref{propsfcs-e}, we can assure the existence and uniqueness of $\delta$ verifying \eqref{eqforlambdaN}, if $ T_{0s} \leq W(0) < W(+\infty) <T_{0l} $, or equivalently \eqref{cotash0}, and the thesis holds.
\end{proof}

\section{An Inequality for the Coefficient $\mu$ of the Rubinstein Free Boundary}\label{equivprob2}
For the solution given in Theorem \ref{teoprincipal2}, the temperature in the fixed face $x=0$ is given by
\begin{equation}\label{temp0N}
\widehat{T}_1 = \widehat{T}_s(0,t)=\widehat{T}_k-\frac{h_0 \sqrt{\pi\al_s}erf(\delta)(\widehat{T}_k - T_\infty)}{k_s+h_0 \sqrt{\pi\al_s}erf(\delta)}.
\end{equation}
Since $\widehat{T}_1<\widehat{T}_k$, we can consider the problem $(P_R)$ defined by \eqref{PS-ABalt}$(i)$-$(iv)$,$(\widetilde{v})$,$(vi)$-$(xiv)$, where
\begin{equation}\label{convcond}
\begin{array}{lll}
\widetilde{v}. & \widehat{T}_s(0,t)=T_1,& t>0.
\end{array}
\end{equation}
which has a unique solution, known as Rubinstein solution, and was given by \eqref{RubSol}-\eqref{condTkN}.
\begin{theorem}\label{equiv1N}
If $T_1=\widehat{T}_1$, under condition \eqref{cotash0} and \eqref{temp0N}, we obtain that the two free boundaries problems $(P_2)$ and $(P_R)$ are equivalents, that is
\begin{equation}\label{condmudelta}
\mu=\delta, \quad \widehat{T}_k=\widetilde{T}_k,
\end{equation}
\begin{equation}\label{TyCtildesN}
\widehat{T}_s(x,t)=\widetilde{T}_s(x,t), \quad \widehat{T}_l(x,t)=\widetilde{T}_l(x,t), \quad \widehat{C}_s(x,t)=\widetilde{C}_s(x,t), \quad \widehat{C}_l(x,t)=\widetilde{C}_l(x,t).
\end{equation}
\end{theorem}
\begin{proof}
By \eqref{eqforlambdaN} we obtain that $M(\delta)=\phi(\widehat{T}_k)$. We prove now that $G(\delta)=\widehat{T}_k$.

Observe that,
\begin{equation}\label{ecequivConv1}
\widehat{T}_k=\frac{k_s \al_l Q\bigg(\frac{\sqrt{\al_s}}{\sqrt{\al_l}}\delta\bigg)\widehat{T}_k + k_l \al_s Q_1(\delta)\widehat{T}_k}{k_s \al_l Q\bigg(\frac{\sqrt{\al_s}}{\sqrt{\al_l}}\delta\bigg) + k_l \al_s Q_1(\delta)}=\frac{A+B}{C}.
\end{equation}
Now, we work with $B$.
\begin{equation}\label{ecequivConv2}
\begin{split}
B&= k_l \al_s Q_1(\delta)\widehat{T}_k\\
&= k_l \al_s Q_1(\delta)\widehat{T}_k\pm \frac{h_0 k_s \al_l \sqrt{\pi\al_s}erf(\delta)}{k_s+h_0\sqrt{\pi\al_s}erf(\delta)}Q\bigg(\frac{\sqrt{\al_s}}{\sqrt{\al_l}}\delta\bigg)\widehat{T}_k\\
&= \frac{k_l \al_s Q_1(\delta)[k_s+h_0\sqrt{\pi\al_s}erf(\delta)]+h_0 k_s \al_l \sqrt{\pi\al_s}erf(\delta)Q\bigg(\frac{\sqrt{\al_s}}{\sqrt{\al_l}}\delta\bigg)}{k_s+h_0\sqrt{\pi\al_s}erf(\delta)}\widehat{T}_k \\
&\quad - \frac{h_0 k_s \al_l \sqrt{\pi\al_s}erf(\delta)}{k_s+h_0\sqrt{\pi\al_s}erf(\delta)}Q\bigg(\frac{\sqrt{\al_s}}{\sqrt{\al_l}}\delta\bigg)\widehat{T}_k\\
&= \frac{k_l \al_s Q_1(\delta)[k_s+h_0\sqrt{\pi\al_s}erf(\delta)]+h_0 k_s \al_l \sqrt{\pi\al_s}erf(\delta)Q\bigg(\frac{\sqrt{\al_s}}{\sqrt{\al_l}}\delta\bigg)}{k_s+h_0\sqrt{\pi\al_s}erf(\delta)}\cdot\\
&\quad\Bigg[\frac{k_l \sqrt{\al_s}\delta [k_s+h_0\sqrt{\pi\al_s}erf(\delta)]}{h_0 k_s \al_l e^{-\delta^2}Q\bigg(\frac{\sqrt{\al_s}}{\sqrt{\al_l}}\delta\bigg) +k_l\sqrt{\al_s}\delta[k_s+h_0\sqrt{\pi\al_s}erf(\delta)]}\cdot  \\
&\quad \Bigg(\frac{h_0 k_s \al_l e^{-\delta^2}Q\bigg(\frac{\sqrt{\al_s}}{\sqrt{\al_l}}\delta\bigg)}{k_l \sqrt{\al_s}\delta[k_s+h_0\sqrt{\pi\al_s}erf(\delta)]}T_\infty + T_0\Bigg) \\
&\quad +\frac{\delta\rho \al_l \sqrt{\al_s}\delta [k_s+h_0\sqrt{\pi\al_s}erf(\delta)]Q\bigg(\frac{\sqrt{\al_s}}{\sqrt{\al_l}}\delta\bigg)}{h_0 k_s \al_l e^{-\delta^2}Q\bigg(\frac{\sqrt{\al_s}}{\sqrt{\al_l}}\delta\bigg) +k_l\sqrt{\al_s}\delta[k_s+h_0\sqrt{\pi\al_s}erf(\delta)]}
\Bigg]\\
&\quad - \frac{h_0 k_s \al_l \sqrt{\pi\al_s}erf(\delta)}{k_s+h_0\sqrt{\pi\al_s}erf(\delta)}Q\bigg(\frac{\sqrt{\al_s}}{\sqrt{\al_l}}\delta\bigg)\widehat{T}_k\\
&= k_l \al_s Q_1(\delta)
\Bigg(\frac{h_0 k_s \al_l e^{-\delta^2}Q\bigg(\frac{\sqrt{\al_s}}{\sqrt{\al_l}}\delta\bigg)}{k_l \sqrt{\al_s}\delta[k_s+h_0\sqrt{\pi\al_s}erf(\delta)]}T_\infty + T_0\Bigg)
 +\gamma \rho \al_l\al_s Q_1(\delta) Q\bigg(\frac{\sqrt{\al_s}}{\sqrt{\al_l}}\delta\bigg)\\
&\quad - \frac{h_0 k_s \al_l \sqrt{\pi\al_s}erf(\delta)}{k_s+h_0\sqrt{\pi\al_s}erf(\delta)}Q\bigg(\frac{\sqrt{\al_s}}{\sqrt{\al_l}}\delta\bigg)\widehat{T}_k\\
&= -\frac{h_0 k_s \al_l \sqrt{\pi\al_s}erf(\delta)}{k_s+h_0\sqrt{\pi\al_s}erf(\delta)}Q\bigg(\frac{\sqrt{\al_s}}{\sqrt{\al_l}}\delta\bigg) (\widehat{T}_k-T_\infty) +  k_l \al_s Q_1(\delta) T_0\\
&\quad +\gamma \rho \al_l\al_s Q_1(\delta) Q\bigg(\frac{\sqrt{\al_s}}{\sqrt{\al_l}}\delta\bigg)\\
\end{split}
\end{equation}
Then, by \eqref{ecequivConv1}, \eqref{ecequivConv2}, and taking into account that $T_1=\widehat{T}_1$ is given by \eqref{temp0N}, we have
\begin{equation}
\begin{split}
\widehat{T}_k&=\frac{k_s \al_l Q\bigg(\frac{\sqrt{\al_s}}{\sqrt{\al_l}}\delta\bigg)\widehat{T}_k + k_l \al_s Q_1(\delta)\widehat{T}_k}{k_s \al_l Q\bigg(\frac{\sqrt{\al_s}}{\sqrt{\al_l}}\delta\bigg) + k_l \al_s Q_1(\delta)}\\
&=\frac{k_s \al_l Q\bigg(\frac{\sqrt{\al_s}}{\sqrt{\al_l}}\delta\bigg)\widehat{T}_k -\frac{h_0 k_s \al_l \sqrt{\pi\al_s}erf(\delta)}{k_s+h_0\sqrt{\pi\al_s}erf(\delta)}Q\bigg(\frac{\sqrt{\al_s}}{\sqrt{\al_l}}\delta\bigg) (\widehat{T}_k-T_\infty)}{k_s \al_l Q\bigg(\frac{\sqrt{\al_s}}{\sqrt{\al_l}}\delta\bigg) + k_l \al_s Q_1(\delta)}\\
&\quad+\frac{k_l \al_s Q_1(\delta) T_0 +\gamma \rho \al_l\al_s Q_1(\delta) Q\bigg(\frac{\sqrt{\al_s}}{\sqrt{\al_l}}\delta\bigg)}{k_s \al_l Q\bigg(\frac{\sqrt{\al_s}}{\sqrt{\al_l}}\delta\bigg) + k_l \al_s Q_1(\delta)}\\
&=\frac{k_s \al_l Q\bigg(\frac{\sqrt{\al_s}}{\sqrt{\al_l}}\delta\bigg)\widehat{T}_1+  k_l \al_s Q_1(\delta) T_0 +\gamma \rho \al_l\al_s Q_1(\delta) Q\bigg(\frac{\sqrt{\al_s}}{\sqrt{\al_l}}\delta\bigg)}{k_s \al_l Q\bigg(\frac{\sqrt{\al_s}}{\sqrt{\al_l}}\delta\bigg) + k_l \al_s Q_1(\delta)}\\
&=G(\delta).
\end{split}
\end{equation}
Thus, by the uniqueness of the pair $\widetilde{T}_k,\mu$ in \eqref{condTkN}, $\mu=\delta$, and $\widetilde{T}_k=\widehat{T}_k$.

On the other hand, taking into account that $\delta=\mu$, we have
\begin{equation*}
\begin{split}
D_s^T(x)&=\mathscr{A}_s^T(x), \quad \forall x>0, \quad\quad\quad D_l^T(x)=\mathscr{A}_l^T(x), \quad \forall x>0,\\
D_s^C(x)&=\mathscr{A}_s^C(x), \quad \forall x>0, \quad\quad\quad D_l^C(x)=\mathscr{A}_l^C(x), \quad \forall x>0,\\
E_s^T(x)&=\mathscr{B}_s^T(x), \quad \forall x>0, \quad\quad\quad E_l^T(x)=\mathscr{B}_l^T(x), \quad \forall x>0,\\
E_s^C(x)&=\mathscr{B}_s^C(x), \quad \forall x>0, \quad\quad\quad E_l^C(x)=\mathscr{B}_l^C(x), \quad \forall x>0,\\
\end{split}
\end{equation*}
that is \eqref{TyCtildesN}.
\end{proof}
\begin{corollary}
Under the hypothesis of Theorem \ref{equiv1N}, the coefficient $\mu$ of the free boundary $\widetilde{s}(t)$ and the initial temperature $\widetilde{T}_k$ of solidification, corresponding to the Rubinstein solution to the free boundary problem $(P_R)$ satisfies the inequalities
\begin{equation}\label{desigserflambdaN}
\frac{\sqrt{ \al_l}}{\sqrt{\al_s}}\frac{k_s}{k_l}\frac{\widetilde{T}_k-T_1}{T_0-T_{0s}}\frac{T_{0s}-T_\infty}{T_1-T_\infty}<erf(\mu)<\frac{\sqrt{ \al_l}}{\sqrt{\al_s}}\frac{k_s}{k_l}\frac{\widetilde{T}_k-T_1}{T_0-T_{0l}}\frac{T_{0l}-T_\infty}{T_1-T_\infty}.
\end{equation}
\end{corollary}
\begin{proof}
From the solution to problem $P_R$ equivalent to problem $P_2$, by using condition \eqref{PS-ABalt}-v, we obtain,
\begin{equation*}
h_0=\frac{k_s}{\sqrt{\al_s \pi} erf(\mu)}\frac{\widetilde{T}_k-T_1}{T_1-T_\infty}.
\end{equation*}
By the equivalence between problems $(P_2)$ and $(P_R)$, $h_0$ must satisfy inequalities \eqref{cotash0}.
Then, the next inequality holds:
\begin{equation}
\frac{(T_0-T_{0l})k_l}{(T_{0l}-T_\infty)\sqrt{\pi \al_l}}<\frac{k_s}{\sqrt{\al_s \pi} erf(\mu)}\frac{\widetilde{T}_k-T_1}{T_1-T_\infty}<\frac{(T_0-T_{0s})k_l}{(T_{0s}-T_\infty)\sqrt{\pi \al_l}},
\end{equation}
or equivalently \eqref{desigserflambdaN}.
\end{proof}
If we wish to obtain an inequality for the coefficient $\mu$ of the Rubinstein solution we must eliminate the dependence on $T_\infty$. Taking the maximum of $\frac{T_{0s}-T_\infty}{T_1-T_\infty}$ respect to $T_\infty$ on the left hand side, and the minimum of $\frac{T_{0l}-T_\infty}{T_1-T_\infty}$ respect to $T_\infty$ on the right hand side in \eqref{desigserflambdaN}, we get
\begin{remark}
\begin{equation}\label{desigserflambdaNlim}
\frac{\sqrt{ \al_l}}{\sqrt{\al_s}}\frac{k_s}{k_l}\frac{\widetilde{T}_k-T_1}{T_0-T_{0s}}<erf(\mu)<\frac{\sqrt{ \al_l}}{\sqrt{\al_s}}\frac{k_s}{k_l}\frac{\widetilde{T}_k-T_1}{T_0-T_{0l}}.
\end{equation}
which is the same inequalities given by \eqref{desigserflambda}. Therefore, the Remark \ref{Remark2} holds.
\end{remark}
\begin{remark}
The results in this section generalizes \cite{Ta:2017}. In fact, if $C(x,t)$ is constant and if we consider only the thermal problem, then condition \eqref{cotash0} is the inequality given in \cite[Theorem 1]{Ta:2017} obtained for the two-phase Stefan problem with a heat flux condition at $x=0$.
\end{remark}

\section{Conclusion}
Two explicit solutions of a similarity type for the temperature and the concentration in the two-phase binary-alloy solidification problem in a semi-infinite material were obtained for two different boundary conditions at the fixed face $x=0$. When the boundary condition is the heat flux condition given by \eqref{PS-AB}(v) then there exists an instantaneous solidification process if and only if the coefficient $q_0$ satisfies the inequalities
\begin{equation*}
\frac{(T_0-T_{0l})k_l}{\sqrt{\pi \al_l}}<q_0<\frac{(T_0-T_{0s})k_l}{\sqrt{\pi \al_l}}.
\end{equation*}
When the boundary condition is a convective condition given by \eqref{PS-ABalt}(v) then there exists an instantaneous solidification process if and only if the coefficient $h_0$ satisfies the inequalities
\begin{equation*}
\frac{(T_0-T_{0l})k_l}{(T_{0l}-T_1)\sqrt{\pi \al_l}}<h_0<\frac{(T_0-T_{0s})k_l}{(T_{0s}-T_1)\sqrt{\pi \al_l}}.
\end{equation*}
Moreover, if $\mu$ is the coefficient of the free boundary corresponding to the Rubinstein solution \eqref{RubSol}-\eqref{condTkN} to the free boundary problem $(P_R)$, then it satisfies the inequalities
\begin{equation}
\frac{\sqrt{ \al_l}}{\sqrt{\al_s}}\frac{\widetilde{T}_k-T_1}{T_0-T_{0s}}\frac{k_s}{k_l}<erf(\mu)<\frac{\sqrt{ \al_l}}{\sqrt{\al_s}}\frac{\widetilde{T}_k-T_1}{T_0-T_{0l}}\frac{k_s}{k_l},
\end{equation}


\section{Acknowledgements}
This work was partially supported by the Projects PIP N$^\circ$ 0275 from CONICET--Universidad Austral (Rosario, Argentina) and European Unions Horizon 2020 research and innovation programme under the Marie Sklodowska-Curie Grant Agreement N$^\circ$ 823731 CONMECH.


\bibliographystyle{plain}

\bibliography{Biblio}

\end{document}